\documentclass[reqno, 11pt]{amsart}
\usepackage{amssymb,verbatim}
\usepackage[nesting]{hyperref}

\usepackage[pdftex]{graphicx}
\usepackage{listings}
\usepackage{multirow}
\usepackage{placeins}
\usepackage{color}
\usepackage{subfigure}
\usepackage{lscape}
\usepackage{dsfont}

\usepackage[T1]{fontenc}
\usepackage[latin9]{inputenc}
\usepackage{xcolor}
\usepackage{pdfcolmk}
\usepackage{units}
\usepackage{mathtools}
\usepackage{amsthm}
\usepackage{amssymb}
\PassOptionsToPackage{normalem}{ulem}
\usepackage{ulem}

\makeatletter

\newcommand{\BlackBoxes}{\global\overfullrule5pt}

\BlackBoxes

\newcommand{\R}{\mathbb{R}}

\numberwithin{equation}{section}
\numberwithin{figure}{section}
\theoremstyle{plain}
\newtheorem{thm}{\protect\theoremname}
\theoremstyle{plain}
\newtheorem{prop}[thm]{\protect\propositionname}
\theoremstyle{definition}
\newtheorem{defn}[thm]{\protect\definitionname}
\theoremstyle{remark}
\newtheorem{rem}[thm]{\protect\remarkname}
\theoremstyle{plain}
\newtheorem{lem}[thm]{\protect\lemmaname}
\theoremstyle{plain}
\newtheorem{cor}[thm]{\protect\corollaryname}

\usepackage{datetime,color,currfile,pdfsync,cancel,hyperref}

\providecommand{\corollaryname}{Corollary}
\providecommand{\definitionname}{Definition}
\providecommand{\lemmaname}{Lemma}
\providecommand{\propositionname}{Proposition}
\providecommand{\remarkname}{Remark}
\providecommand{\theoremname}{Theorem}

\def\0{\kern0pt\-\nobreak\hskip0pt\relax}

\makeatletter
\AtBeginDocument{%
 \def\@serieslogo{%
 \vbox to\headheight{%
 \parindent\z@ \fontsize{6}{7\p@}\selectfont
 \vss}}}

\def\makeoverbar#1#2#3#4#5#6#7{%
 \setbox0=\hbox{$\m@th#2\mkern#5mu{{}#3{}}\mkern#6mu$}%
 \setbox1=\null \dimen@=#4\fontdimen8#13 \dimen@=3.5\dimen@
 \advance\dimen@ by \ht0 \dimen@=-#7\dimen@ \advance\dimen@ by \wd0
 \ht1=\ht0 \dp1=\dp0 \wd1=\dimen@
 \dimen@=\fontdimen8#13 \fontdimen8#13=#4\fontdimen8#13
 \rlap{\hbox to \wd0{$\m@th\hss#2{\overline{\box1}}\mkern#5mu$}}
 \fontdimen8#13=\dimen@}

\def\mylabel#1#2{{\def\@currentlabel{#2}\label{#1}}}

\begin{document}



\title[Representation of ggBM]{Integral Representation of\\
	Generalized Grey Brownian Motion}

\author[W. \smash{Bock}]{Wolfgang Bock${}^*$}
\address[W. Bock]{Department of Mathematics,
TU Kaiserslautern (TUK), D-67663 Kaiserslautern, Germany}

\email{\href{mailto:bock@mathematik.uni-kl.de}
{bock@mathematik.uni-kl.de}}

\author[S. \smash{Desmettre}]{Sascha Desmettre${}^*$,$^\ddagger$ }
\address[S. Desmettre]{Department of Mathematics,
TU Kaiserslautern (TUK), D-67663 Kaiserslautern, Germany and Institute for Mathematics and Scientific Computing, University of Graz, Heinrichstra\ss{}e 36, AT-8010 Graz, Austria}

\email{\href{desmettre@mathematik.uni-kl.de}{desmettre@mathematik.uni-kl.de}  \href{sascha.desmettre@uni-graz.at}{sascha.desmettre@uni-graz.at}}

\author[J.L. \smash{da Silva}]{Jos{\'e} Lu{\'\i}s
da Silva${}^\dagger$}
\address[J.L. da Silva]{CIMA, University of Madeira, Campus da Penteada, 9020-105 Funchal, Portugal}

\email{\href{mailto:joses@staff.uma.pt} {joses@staff.uma.pt}}

\thanks{${}^*$ Department of Mathematics, TU Kaiserslautern (TUK), D-67663 Kaiserslautern, Germany}
\thanks{${}^\dagger$ CIMA, University of Madeira, Campus da Penteada, 9020-105 Funchal, Portugal}
\thanks{${}^\ddagger$ Institute for Mathematics and Scientific Computing, University of Graz, AT-8010 Graz, Austria}

\begin{abstract}
In this paper we investigate the representation of a class of non
Gaussian processes, namely generalized grey Brownian motion, in terms
of a weighted integral of a stochastic process which is a solution of
a certain stochastic differential equation. In particular the underlying process can be seen as a non Gaussian extension of the Ornstein-Uhlenbeck process, hence generalizing the representation results of {Muravlev, Russian Math. Surveys, 66(2), 2011} as well as {Harms and Stefanovits, Stochastic Process.~Appl., 129, 2019} to the non Gaussian case. 
\end{abstract}
\maketitle

\vspace{0.5cm}
\begin{minipage}{14cm}
{\small
\begin{description}
\item[\rm \textsc{ Key words} ]
{\small Generalized grey Brownian motion; Fractional stochastic processes; Rough paths; superposition of OU processes; grey OU processes}
\end{description}
}
\end{minipage}

\section{Introduction}
\label{sec:introduction}

In recent years, as an extension of Brownian motion (Bm), fractional Brownian motion (fBm) has become an object of intensive study \cite{oks}, \cite{mishura08}, due to its specific properties, such as short/long range dependence and self-similarity, with natural applications in different fields (e.g. mathematical finance, telecommunications engineering, etc.). 
{In order to cast fBm into the classical Bm framework, t}here are various representations of fBm, starting with the famous definition by Mandelbrot and van Ness \cite{MandelbrotNess1968}. {This idea is also the starting point} for a characterization of fBm using an infinite superposition of Ornstein-Uhlenbeck processes w.r.t.~the standard Wiener process; compare the works of Carmona, Coutin, Montseny, and Muravlev \cite{CC1993,CCM2000,MU2011} or also the monograph of \cite{mishura08}. Recently, further applications of this representation have for instance been investigated in \cite{HS2018} with a focus on finance and in \cite{BD2018} in the context of optimal portfolios. 

One key tool box for the rigorous analysis of fBm is the Gaussian analysis or white noise analysis. 
White noise analysis has evolved into an infinite dimensional distribution theory, with rapid developments in mathematical structure and applications in multiple domains, see e.g.~the monographs \cite{HKPS93, O94, Kuo96, HS17}. Various characterization theorems \cite{PS91, KLPSW96, GKS97} are proven to build up a strong analytical foundation. 
Almost at the same time, first attempts were made to introduce a non-Gaussian infinite dimensional analysis, by transferring properties of the Gaussian measure to the Poisson measure \cite{I88}, which could be generalized with the help of a biorthogonal  generalized Appell systems \cite{Da91, ADKS96, KSWY95}.
Mittag-Leffler Analysis is established in \cite{GJRS14} and \cite{GJ15}. In fact, it generalizes methods from white noise calculus to the case, where in the characteristic function of the Gaussian measure the exponential function is replaced by a Mittag-Leffler function. The corresponding {stochastic process is referred to as} generalized grey Brownian motion (ggBm) and is in general neither a martingale nor a Markov process. Moreover, it is not possible - as in the Gaussian case - to find a proper orthonormal system of polynomials for the test and generalized functions. Here it is necessary to make use of an Appell system of biorthogonal polynomials.
The grey noise measure \cite{Schneider90, Schneider92, MM09} is included as a special case in the class of Mittag-Leffler measures, which offers the possibility to apply the Mittag-Leffler analysis to fractional differential equations, in particular to fractional diffusion equations, which carry numerous applications in science, like relaxation type differential equations or viscoelasticity. In \cite{GJ15} also a relation between the heat kernel in this setting and the associated processes grey Brownian motion could be proven. In \cite{BS17} Wick-type stochastic differential equations and Ornstein-Uhlenbeck processes were solved within the framework of Mittag-Leffler analysis.
Indeed the underlying fractional differential equations are of interest in applications like human mobility in disease spread \cite{SMGABS2012}.

The aim of this paper is to establish a link between ggBm and generalized grey Ornstein-Uhlenbeck processes as an extension of the results in \cite{BS17} {and to extend the representation results of \cite{MU2011} and \cite{HS2018} to the non-gaussian case of ggBm}. To this end, we make use of a representation of ggBm as a product of a positive and time-independent random variable and a fBm \cite{Mainardi_Mura_Pagnini_2010}. {Our work also shares common features with \cite{Pagnini18}.} 

The manuscript is organized as follows: In Section 2 we give preliminaries about generalized grey Brownian motion{, defined on an abstract probability space.} In Section 3 we develop a representation of ggBm using an infinite dimensional superposition of generalized grey Ornstein-Uhlenbeck processes. We thus enhance the results in \cite{CC1993,CCM2000,MU2011, HS2018} to the setting of ggBm. {The Appendix contains auxiliary results needed in the two main proofs.}

\section{Generalized grey Brownian motion in arbitrary dimensions}
\label{sec:ggBm_high_d}

\subsection{Prerequisites}

We define the operator $M_{-}^{\alpha/2}$ on {the Schwartz test function space} $S(\mathbb{R})$ by
\[
M_{-}^{\alpha/2}\varphi:=\begin{cases}
K_{\alpha/2}D_{-}^{-\nicefrac{(\alpha-1)}{2}}\varphi, & \alpha\in(0,1),\\
\varphi, & \alpha=1,\\
K_{\alpha/2}I_{-}^{\nicefrac{(\alpha-1)}{2}}\varphi, & \alpha\in(1,2),
\end{cases}
\]
with the normalization constant $K_{\alpha/2}:=\sqrt{\alpha\sin(\nicefrac{\alpha\pi}{2})\Gamma(\alpha)}$
. $D_{-}^{r}$, $I_{-}^{r}$ denote the left-side fractional derivative
and fractional integral of order $r$ in the sense of Riemann-Liouville,
respectively:
\begin{eqnarray*}
(D_{-}^{r}f)(x) & = & \frac{1}{\Gamma(1-r)}\frac{d}{dx}\int_{-\infty}^{x}f(t)(x-t)^{-r}dt,\\
(I_{-}^{r}f)(x) & = & \frac{1}{\Gamma(r)}\int_{x}^{\infty}f(t)(t-x)^{r-1}dt,\hfill x\in\mathbb{R}.
\end{eqnarray*}
We refer to \cite{SKM1993} or \cite{KST2006} for the details on
these operators. It is possible to obtain a larger domain of the operator
$M_{-}^{\nicefrac{\alpha}{2}}$ in order to include the indicator
function $1\!\!1_{[0,t)}$ such that $M_{-}^{\nicefrac{\alpha}{2}}1\!\!1_{[0,t)}\in L^{2}$,
for the details we refer to Appendix A in \cite{GJ15}. We have the
following
\begin{prop}[Corollary 3.5 in \protect\cite{GJ15}]
For all $t,s\ge0$ and all $0<\alpha<2$ it holds that 
\begin{equation}
\big(M_{-}^{\nicefrac{\alpha}{2}}1\!\!1_{[0,t)},M_{-}^{\nicefrac{\alpha}{2}}1\!\!1_{[0,s)}\big)_{L^{2}(\mathbb{R})}=\frac{1}{2}\big(t^{\alpha}+s^{\alpha}-|t-s|^{\alpha}\big).\label{eq:alpha-inner-prod}
\end{equation}
\end{prop}

The Mittag-Leffler function was introduced by G.\ Mittag-Leffler
in a series of papers \cite{Mittag-Leffler1903,Mittag-Leffler1904,Mittag-Leffler1905}.
\begin{defn}[Mittag-Leffler function]
\label{def:MLf}
\begin{enumerate}
\item For $\beta>0$ the Mittag-Leffler function $E_{\beta}$ (MLf for short)
is defined as an entire function by the following series representation
\begin{equation}
E_{\beta}(z):=\sum_{n=0}^{\infty}\frac{z^{n}}{\Gamma(\beta n+1)},\quad z\in\mathbb{C},\label{eq:MLf}
\end{equation}
where $\Gamma$ denotes the gamma function.
\item For any $\rho\in\mathbb{C}$ the generalized Mittag-Leffler function
(gMLf for short) is an entire function defined by its power series
\begin{equation}
E_{\beta,\rho}(z):=\sum_{n=0}^{\infty}\frac{z^{n}}{\Gamma(\beta n+\rho)},\quad z\in\mathbb{C}.\label{eq:gMLf}
\end{equation}
Note the relation $E_{\beta,1}(z)=E_{\beta}(z)$ and $E_{1}(z)=e^{z}$
for any $z\in\mathbb{C}$. 
\end{enumerate}
\end{defn}

The Wright function is defined by the following series representation
which is convergent in the whole $z$-complex plane
\[
W_{\lambda,\mu}(z):=\sum_{n=0}^{\infty}\frac{z^{n}}{n!\Gamma(\lambda n+\mu)},\quad\lambda>-1,\;\mu\in\mathbb{C}.
\]
An important particular case of the Wright function is the so called
$M$-Wright function $M_{\beta}$, $0<\beta\le1$ (in one variable)
defined by 
\begin{equation}
\label{mbeta_series}
M_{\beta}(z):=W_{-\beta,1-\beta}(-z)=\sum_{n=0}^{\infty}\frac{(-z)^{n}}{n!\Gamma(-\beta n+1-\beta)}.
\end{equation}
For the choice $\beta=\nicefrac{1}{2}$ the corresponding $M$-Wright
function reduces to the Gaussian density
\begin{equation}
M_{\nicefrac{1}{2}}(z)=\frac{1}{\sqrt{\pi}}\exp\left(-\frac{z^{2}}{4}\right).\label{eq:MWright_Gaussian}
\end{equation}

The MLf $E_{\beta}$ and the $M$-Wright are related through the Laplace
transform
\begin{equation}
\int_{0}^{\infty}e^{-s\tau}M_{\beta}(\tau)\,d\tau=E_{\beta}(-s).\label{eq:LaplaceT_MWf}
\end{equation}

The $\mathbb{M}$-Wright function with two variables $\mathbb{M}_{\beta}^{1}$
of order $\beta$ (1-dimension in space) is defined by
\begin{equation}
\mathbb{M}_{\beta}^{1}(x,t):=\mathbb{M}_{\beta}(x,t):=\frac{1}{2}t^{-\beta}M_{\beta}(|x|t^{-\beta}),\quad0<\beta<1,\;x\in\mathbb{R},\;t\in\mathbb{R}^{+}\label{eq:MWf-2variables}
\end{equation}
which defines a spatial probability density in $x$ evolving in time
$t$ with self-similarity exponent $\beta$. The following integral
representation for the $\mathbb{M}$-Wright is valid, see \cite{Mainardi:2003vn}.
\begin{equation}
\mathbb{M}_{\nicefrac{\beta}{2}}(x,t)=2\int_{0}^{\infty}\frac{e^{-\nicefrac{x^{2}}{4\tau}}}{\sqrt{4\pi\tau}}t^{-\beta}M_{\beta}(\tau t^{-\beta})\,d\tau,\quad0<\beta\leq1,\;x\in\mathbb{R}.\label{eq:subordination_MWf1}
\end{equation}
This representation is valid in more general form, see \cite[eq.~(6.3)]{Mainardi:2003vn},
but for our purpose it is sufficient in view of its generalization
for $x\in\mathbb{R}^{d}$. In fact, eq.\ (\ref{eq:subordination_MWf1})
may be extended to a general spatial dimension $d$ by the extension
of the Gaussian function, namely
\begin{equation}
\mathbb{M}_{\nicefrac{\beta}{2}}^{d}(x,t):=2\int_{0}^{\infty}\frac{e^{-\nicefrac{|x|^{2}}{4\tau}}}{(4\pi\tau)^{\nicefrac{d}{2}}}t^{-\beta}M_{\beta}(\tau t^{-\beta})\,d\tau,\quad x\in\mathbb{R}^{d},\;t\ge0,\;0<\beta\le1.\label{eq:MWf_d_dimension}
\end{equation}
The function $\mathbb{M}_{\nicefrac{\beta}{2}}^{d}$ is nothing but
the density of the fundamental solution of a time-fractional diffusion
equation, see \cite{Mentrelli2015}.

\subsection{{Generalized grey Brownian motion}}

{
\begin{defn}[see \cite{MM09} for $d=1$]
Let $0<\beta<1$ and $0<\alpha<2$ be given. A $d$-dimensional continuous
stochastic process $B^{\beta,\alpha}=\{B^{\beta,\alpha}(t),\,t\geq0\}$
defined on a complete probability space $\left(\Omega,\mathcal{F},P\right)$
is a generalized grey Brownian motion (ggBm) if:
\begin{enumerate}
\item $P(B^{\beta,\alpha}(0)=0)=0$, that is ggBm starts at zero almost
surely.
\item Any collection $\big\{ B^{\beta,\alpha}(t_{1}),\ldots,B^{\beta,\alpha}(t_{n})\big\}$
with $0\leq t_{1}<t_{2}<\ldots<t_{n}<\infty$ has characteristic function
given, for any $\theta=(\theta_{1},\ldots,\theta_{n})\in(\mathbb{R}^{d})^{n}$,
by 
\begin{equation}
\mathbb{E}\left(\exp\left(i\sum_{k=1}^{n}(\theta_{k},B^{\beta,\alpha}(t_{k}))\right)\right)=E_{\beta}\left(-\frac{1}{2}\sum_{k=1}^{n}(\theta_{k},\Sigma_{\alpha}\theta_{k})\right),\label{eq:charact-func-ggBm}
\end{equation}
where 
\[
\Sigma_{\alpha}=\big(t_{k}^{\alpha}+t_{j}^{\alpha}-|t_{k}-t_{j}|^{\alpha}\big)_{k,j=1}^{d}.
\]
\item The joint probability density function of $(B^{\beta,\alpha}(t_{1}),\ldots,B^{\beta,\alpha}(t_{n}))$
is equal to 
\begin{equation}
f_{\beta}(\theta,\Sigma_{\alpha})=\frac{(2\pi)^{-\frac{nd}{2}}}{\sqrt{\det\Sigma_{\alpha}}}\int_{0}^{\infty}\tau^{-\frac{nd}{2}}e^{-\frac{1}{2\tau}\sum_{k=1}^{n}(\theta_{k},\Sigma_{\alpha}^{-1}\theta_{k})}M_{\beta}(\tau)\,d\tau.\label{eq:fdd-density-ggBm}
\end{equation}
\end{enumerate}
\end{defn}}
\begin{rem}
\label{rem:self-similar}The family $\{B^{\beta,\alpha}(t),\;t\geq0,\,\beta\in(0,1],\,\alpha\in(0,2)\}$
forms a class of $\alpha/2$-self-similar processes with stationary
increments which includes:
\begin{enumerate}
\item for $\beta=\alpha=1$, the process $\{B^{1,1}(t),\;t\geq0\}$, standard
$d$-dimensional Bm.
\item for $\beta=1$ and $0<\alpha<2$, $\{B^{1,\alpha}(t),\;t\geq0\}$,
$d$-dimensional fBm with Hurst parameter $\alpha/2$.
\item for $\alpha=1$, $\{B^{\beta,1}(t),\;t\geq0\}$ a $1/2$-self-similar
non Gaussian process with 
\begin{equation}
\emph{E}\left(e^{i\left(k,B^{\beta,1}(t)\right)}\right)=E_{\beta}\left(-\frac{|k|^{2}}{2}t\right),\quad k\in\mathbb{R}^{d}.\label{eq:ch-fc-1/2sssi}
\end{equation}
\item for $0<\alpha=\beta<1$, the process $\{B^{\beta}(t):=B^{\beta,\beta}(t),\;t\geq0\}$,
$\beta/2$-self-similar and called $d$-dimensional grey Brownian
motion (gBm for short). The characteristic function of\ $B^{\beta}(t)$
is given by 
\begin{equation}
\emph{E}\left(e^{i\left(k,B^{\beta}(t)\right)}\right)=E_{\beta}\left(-\frac{|k|^{2}}{2}t^{\beta}\right),\quad k\in\mathbb{R}^{d}.\label{eq:ch-fc-gBm}
\end{equation}
For $d=1$, gBm was introduced by W.\ Schneider in \cite{Schneider90,Schneider92}.
\item For other choices of the parameters $\beta$ and $\alpha$ we obtain non Gaussian processes. 
\end{enumerate}
\end{rem}

It was shown in \cite{Mura_Pagnini_08} ({for $d=1$}) that the ggBm $B^{\beta,\alpha}$
admits the following representation 
\begin{equation}\label{sqrtrepggBm}
\big\{ B^{\beta,\alpha}(t),\;t\geq0\big\}\overset{d}{=}\big\{\sqrt{Y_{\beta}}B^{\nicefrac{\alpha}{2}}(t),\;t\geq0\big\},
\end{equation}
where $\overset{d}{=}$ denotes the equality of the finite dimensional
distribution and $B^{\nicefrac{\alpha}{2}}$ is a standard fBm with
Hurst parameter $H=\alpha/2$. $Y_{\beta}$ is an independent non-negative
random variable with probability density function $M_{\beta}(\tau)$,
$\tau\geq0$. {The proof of \eqref{sqrtrepggBm} for an arbitrary dimension $d$ is a straightforward adaptation of this result.}
\section{Integral Representation of Generalized Grey Brownian Motion}

In this section we put together the representation \eqref{sqrtrepggBm} of ggBm and the representation of fBm as a stochastic integral from Mandelbrot-van Ness \cite{MandelbrotNess1968}, see also \cite{mishura08}, in order to obtain an integral representation for ggBm. The idea to express the fractional integral in the Madelbrot-van Ness representation goes back to \cite{CC1993,CCM2000,MU2011} and has been used to obtain an affine representation of fractional processes in \cite{HS2018}.

We recall from \cite{mishura08} the following result:

\begin{cor}[{cf. \cite[Cor.~1.3.3]{mishura08}}]
For any $H\in(0,1)$ the process
\[
B^{H}(t)=\int_{\mathbb{R}}(M_{-}^{H}1\!\!1_{(0,t)})(s)\,dW(s)=\frac{C_{H}}{\Gamma(H+\frac{1}{2})}\int_{\mathbb{R}}\big((t-s)_{+}^{H-\frac{1}{2}}-(-s)_{+}^{H-\frac{1}{2}}\big)\,dW(s),
\]
is a normalized fBm. Here $W=\{W(t),\,t\in\mathbb{R}\}$ is a two-sided
Wiener process, i.e. the Gaussian process with independent increments
satisfying $\mathbb{E}(W(t))=0$ and $\mathbb{E}(W(t)W(s))=t\wedge s$
for any $s,t\in\mathbb{R}$.
\end{cor}

Together with~\eqref{sqrtrepggBm}, this gives the following representation of ggBm. We first put emphasis on the rough case $0 < \alpha < 1$.

\begin{thm}[Representation of ggBm via Ornstein-Uhlenbeck processes for $0 <\alpha < 1$]\label{repthm1} {Suppose that $\int_0^\infty \left(1\land x^{-\frac{1}{2}} \right) \, \frac{dx}{x^{H+\frac{1}{2}}} < \infty$.} Then, 
	for $0 <\alpha < 1$ and $0< \beta <1$ generalized grey Brownian motion can be represented in finite dimensional distributions as 
	$$
	B^{\beta, \alpha}(t) \stackrel{d}{=} \frac{\cos(\frac{\alpha}{2}\pi)}{\pi}\int_{0}^{\infty}\sqrt{Y_{\beta}}X_{x}(t)\frac{dx}{x^{\frac{1+\alpha}{2}}},
	$$
	where $X_x(t)$ is an Ornstein-Uhlenbeck process w.r.t.~a Brownian motion $W$, i.e. $X_x$ obeys the stochastic differential equation:
	\begin{align}\label{eq:OU_rough}
		dX_x(t) = -x X_x(t) dt + dW(t).
	\end{align}
\end{thm}

\begin{proof}
Note that again we use $\alpha=2H$. 
Due to P. L\'{e}vy \cite{Levy1953} fBm admits the moving average of $W$:
\[
B^{H}(t)=\frac{1}{\Gamma(H+\frac{1}{2})}\int_{0}^{t}(t-s)^{H-\frac{1}{2}}\,dW(s).
\]
For $0<H< \frac{1}{2}$ we use the fact that
\[
(t-s)^{H-\frac{1}{2}}=\frac{1}{\Gamma(\frac{1}{2}-H)}\int_{0}^{\infty}\frac{e^{-x(t-s)}}{x^{H+\frac{1}{2}}}\,dx\,.
\]
{Thus, we arrive at 
\[
B^{H}(t)=\frac{1}{\Gamma(H+\frac{1}{2}) \, \Gamma(\frac{1}{2}-H) }\int_{0}^{t}\left(\int_{0}^{\infty}e^{-x(t-s)}\frac{dx}{x^{H+\frac{1}{2}}}\right)\,dW(s).
\]
By the stochastic Fubini Theorem~\ref{thm:Fubini} we then obtain
\[
B^{H}(t)=\frac{1}{\Gamma(H+\frac{1}{2}) \, \Gamma(\frac{1}{2}-H) }\int_{0}^{\infty}\left(\int_{0}^{t}e^{-x(t-s)} dW(s)\right)\,\frac{dx}{x^{H+\frac{1}{2}}}\,,
\]
where we have chosen $\mu(dx):= \frac{dx}{x^{H+\frac{1}{2}}}$. Condition~\eqref{eq:Fub2} of Theorem~\ref{thm:Fubini} is thereby satisfied, since 
\begin{align*}
\int_0^\infty \sqrt{\int_0^t e^{-2x(t-s)} \, ds} \,\mu(dx) = \int_0^\infty \sqrt{\frac{1-e^{-2xt}}{2x}} \,\mu(dx) \leq \int_0^\infty \sqrt{\frac{1-e^{-2xt}}{x}} \,\mu(dx) < \infty\,,
\end{align*}
where we have used~\eqref{eq:integrability1} of Lemma~\ref{lem:basic} and $\int_0^\infty \left(1\land x^{-\frac{1}{2}} \right) \, \mu(dx) = \int_0^\infty \left(1\land x^{-\frac{1}{2}} \right) \, \frac{dx}{x^{H+\frac{1}{2}}}  < \infty$.
}

\noindent Finally, by the Euler reflection formula
\[
\frac{1}{\Gamma(H+\frac{1}{2})\Gamma(\frac{1}{2}-H)}=\frac{\cos(\pi H)}{\pi}
\]
we may write $B^{H}(t)$, $t\ge0$ as
\begin{eqnarray*}
B^{H}(t) & = & \frac{\cos(\pi H)}{\pi}\int_{0}^{\infty}\left(\int_{0}^{t}e^{-x(t-s)}\,dW(s)\right)\frac{dx}{x^{H+\frac{1}{2}}}\\
 & =: & \frac{\cos(\pi H)}{\pi}\int_{0}^{\infty}X_{x}(t)\frac{dx}{x^{H+\frac{1}{2}}},
\end{eqnarray*}
where $X_{x}(t)$ is a Ornstein-Uhlenbeck process. Hence, using the
representation \eqref{sqrtrepggBm} we obtain the representation in finite
dimensional distribution for ggBm as
$$
B^{\beta,\alpha}(t)=\frac{\cos(\frac{\alpha}{2}\pi)}{\pi}\int_{0}^{\infty}\sqrt{Y_{\beta}}X_{x}(t)\frac{dx}{x^{\frac{1+\alpha}{2}}}.$$
\qedhere
\end{proof}

\begin{cor}[Representation via ggOU processes]
		{Suppose that $\int_0^\infty \left(1\land x^{-\frac{1}{2}} \right) \, \frac{dx}{x^{H+\frac{1}{2}}} < \infty$.} Then, for $0 <\alpha < 1$ and $0< \beta <1$ generalized grey Brownian motion can be represented in finite dimensional distributions as 
		$$
		B^{\beta, \alpha}(t) \stackrel{d}{=}\frac{\cos(\pi\frac{\alpha}{2})}{\pi}\int_{0}^{\infty}Z_{x}^{\beta}(t)\frac{dx}{x^{\frac{1+\alpha}{2}}},$$
		where $Z_{x}^{\beta}(t):=\sqrt{Y_{\beta}}X_{x}(t)$,
		$t\ge0$ satisfies the following stochastic differential equation
		\[
		dZ_{x}^{\beta}(t)=-xZ_{x}^{\beta}(t)\,dt+dB^{\beta,1}(t).
		\]
\end{cor}

The process $Z_x^{\beta}$ is a generalization of the Ornstein-Uhlenbeck process defined in \cite{BS17}. There the authors considered the case $\alpha=\beta$. The generalization however is obvious. Indeed one can compute the characteristic function of $Z_x^{\beta}$ and hence by inverse Fourier transform its probability density function.

\begin{prop}\label{OU:rep1}
The process solving
	\[
	dZ_{x}^{\beta}(t)=-xZ_{x}^{\beta}(t)\,dt+dB^{\beta,1}(t).
	\]
	has the characteristic function 
	$$\mathbb{E}\left(e^{i(k,Z_{x}^{\beta}(t))}\right)= E_{\beta}\left(-\frac{|k|^{2}}{2}\left|f_{x}(t,\cdot)\right|^{2}\right),$$
	where $f_{x}(t,s)=1\!\!1_{[0,t)}(s)e^{-x(t-s)}$.
	Moreover its density function is given by
	$$\rho_{x}(y,t):=\rho_{Z_{x}^{\beta}(t)}(y)=\frac{1}{2}(4\pi)^{\nicefrac{d}{2}}\mathbb{M}_{\nicefrac{\beta}{2}}^{d}\left(\sqrt{2}y,(\frac{1}{2x}(1-e^{-2 xt}))^{\nicefrac{1}{\beta}}\right),\quad y \in \mathbb{R}^d.$$
\end{prop}

\begin{proof}
The characteristic function of $Z_{x}^{\beta}(t)$,
$t\ge0$, $x\ge0$ for any $k\in\mathbb{R}^{d}$ yields
\begin{eqnarray*}
\mathbb{E}\left(e^{i(k,Z_{x}^{\beta}(t))}\right) & = & \int_{0}^{\infty}\mathbb{E}\left(e^{i\sqrt{\tau}(k,X_{x}(t))}\right)M_{\beta}(\tau)\,d\tau\\
 & = & \int_{0}^{\infty}e^{-\frac{|k|^{2}\tau}{2}\left|f_{x}(t,\cdot)\right|^{2}}M_{\beta}(\tau)\,d\tau\\
 & = & E_{\beta}\left(-\frac{|k|^{2}}{2}\left|f_{x}(t,\cdot)\right|^{2}\right),
\end{eqnarray*}
where $f_{x}(t,s)=1\!\!1_{[0,t)}(s)e^{-x(t-s)}$. \\

The density $\rho_{x}(y,t):=\rho_{Z_{x}^{\beta}(t)}(y)$,
$y\in\mathbb{R}^{d}$ of the process $Z_{x}^{\beta}(t)$ may be computed
by an inverse Fourier transform. More precisely, for any $y\in\mathbb{R}^{d}$
\begin{eqnarray*}
\rho_{x}(y,t) & = & \frac{1}{(2\pi)^{\nicefrac{d}{2}}}\int_{\mathbb{R}^{d}}e^{-i(k,y)}E_{\beta}\left(-\frac{|k|^{2}}{2}\left|f_{x}(t,\cdot)\right|^{2}\right)dk\\
 & = & \frac{1}{(2\pi)^{\nicefrac{d}{2}}}\int_{0}^{\infty}M_{\beta}(\tau)\int_{\mathbb{R}^{d}}e^{-i(k,y)-\frac{|k|^{2}}{2}\tau\left|f_{x}(t,\cdot)\right|^{2}}dk.
\end{eqnarray*}
Solving the Gaussian integral yields
\begin{eqnarray*}
\rho_{x}(y,t) & = & \frac{1}{\tau^{\nicefrac{d}{2}}\left|f_{x}(t,\cdot)\right|^{d}}\int_{0}^{\infty}M_{\beta}(\tau)e^{-\frac{|y|^{2}}{2\tau\left|f_{x}(t,\cdot)\right|^{2}}}d\tau
\end{eqnarray*}
making the change of variable $\tilde{\tau}=\tau\left|f_{x}(t,\cdot)\right|^{2}$
and rearranging, we obtain 
\begin{eqnarray*}
\rho_{x}(y,t) & = & \frac{1}{\tau^{\nicefrac{d}{2}}}\int_{0}^{\infty}e^{-\frac{|\sqrt{2}y|^{2}}{4\tau}}\left|f_{x}(t,\cdot)\right|^{-2}M_{\beta}\left(\tau\left|f_{x}(t,\cdot)\right|^{-2}\right)d\tau\\
 & = & (4\pi)^{\nicefrac{d}{2}}\int_{0}^{\infty}\frac{e^{-\frac{|\sqrt{2}y|^{2}}{4\tau}}}{(4\pi\tau)^{\nicefrac{d}{2}}}\mathbb{M}_{\beta}\left(\tau,\left|f_{x}(t,\cdot)\right|^{\nicefrac{2}{\beta}}\right)d\tau\\
 & = & \frac{1}{2}(4\pi)^{\nicefrac{d}{2}}\mathbb{M}_{\nicefrac{\beta}{2}}^{d}\left(\sqrt{2}y,\left|f_{x}(t,\cdot)\right|^{\nicefrac{2}{\beta}}\right)\\
 &=& \frac{1}{2}(4\pi)^{\nicefrac{d}{2}}\mathbb{M}_{\nicefrac{\beta}{2}}^{d}\left(\sqrt{2}y,(\frac{1}{2x}(1-e^{-2 xt}))^{\nicefrac{1}{\beta}}\right).
\end{eqnarray*}
\qedhere
\end{proof}

Consider now the case $1<\alpha<2$, which corresponds to the fractional case with smoother paths.

\begin{thm}[Representation of ggBm via Ornstein-Uhlenbeck processes for $1 <\alpha < 2$]\label{repthm2} {Suppose that $\int_0^\infty \left(1 \land x^{-\frac{3}{2}} \right) \, \frac{dx}{x^{H-\frac{1}{2}}} < \infty$.} Then, for $1 <\alpha < 2$ and $0< \beta <1$ generalized grey Brownian motion can be represented in finite dimensional distributions as 
	$$
	B^{\beta, \alpha} \stackrel{d}{=} \frac{1}{\Gamma(H+\frac{1}{2})\Gamma(\frac{3}{2}-H)}\int_{0}^{\infty}\sqrt{Y_{\beta}}Q_{x}(t)\frac{dx}{x^{\frac{\alpha-1}{2}}},$$
	with $Q_x(t)$ obeying the equation 
	
	$$dQ_x(t) = (-xQ_x(t) + X_{x}(t)) dt, $$
where $X_{x}(t)$
	is an Ornstein-Uhlenbeck process w.r.t.~a Brownian motion. 
\end{thm}

\begin{proof}
	Note that again we use $\alpha=2H$ and the moving average representation.
	
	\[
	B^{H}(t)=\frac{1}{\Gamma(H+\frac{1}{2})}\int_{0}^{t}(t-s)^{H-\frac{1}{2}}\,dW(s).
	\]
	
Now however for $\frac{1}{2}<H<1$ 
$$(t-s)^{H-\frac{1}{2}}$$ can not we written as a Laplace transform anymore and we use:
	\[
	(t-s)^{H-\frac{1}{2}}=\frac{(t-s)}{\Gamma(\frac{3}{2}-H)}\int_{0}^{\infty}(t-s)\frac{e^{-x(t-s)}}{x^{H-\frac{1}{2}}}\,dx.
	\]	
{Thus, we now arrive at 
\[
B^{H}(t)=\frac{1}{\Gamma(H+\frac{1}{2}) \, \Gamma(\frac{3}{2}-H) }\int_{0}^{t}\left(\int_{0}^{\infty} (t-s) \, e^{-x(t-s)}\frac{dx}{x^{H-\frac{1}{2}}}\right)\,dW(s).
\]
By the stochastic Fubini Theorem~\ref{thm:Fubini} we then obtain
\[
B^{H}(t)=\frac{1}{\Gamma(H+\frac{1}{2}) \, \Gamma(\frac{3}{2}-H) }\int_{0}^{\infty}\left(\int_{0}^{t} (t-s) \, e^{-x(t-s)} dW(s)\right)\,\frac{dx}{x^{H-\frac{1}{2}}}\,,
\]
where we have chosen $\nu(dx):= \frac{dx}{x^{H-\frac{1}{2}}}$. Condition~\eqref{eq:Fub2} of Theorem~\ref{thm:Fubini} is thereby satisfied, since 
\begin{align*}
\int_0^\infty \sqrt{\int_0^t (t-s)^2\, e^{-2x(t-s)} \, ds} \,\nu(dx) &= \int_0^\infty \sqrt{\frac{1-e^{-2xt}(1+2tx+2t^2x^2)}{4x^3}} \,\nu(dx) \\
&\leq \int_0^\infty \sqrt{\frac{1-e^{-2xt}(1+2tx+2t^2x^2)}{x^3}} \,\nu(dx) < \infty\,,
\end{align*}
where we have used integration by parts and~\eqref{eq:integrability2} of Lemma~\ref{lem:basic} as well as the fact that $\int_0^\infty \left(1 \land x^{-\frac{3}{2}} \right) \, \nu(dx) = \int_0^\infty \left(1 \land x^{-\frac{3}{2}} \right) \, \frac{dx}{x^{H-\frac{1}{2}}} < \infty$.
}Now we can write
	\begin{equation}\label{eq:Qx}
	Q_x(t)=\int_{0}^{t}(t-s)e^{-x(t-s)}\,dW(s)\,,
	\end{equation}
	implying
    $$dQ_x(t) = (-xQ_x(t) + X_{x}(t)) dt, $$
	where $X_{x}(t)$ is again the Ornstein-Uhlenbeck process~\eqref{eq:OU_rough}. Hence, using the
	representation \eqref{sqrtrepggBm} we obtain the representation in finite
	dimensional distribution for ggBm as
	$$
	B^{\beta,\alpha}(t)=\frac{1}{\Gamma(H+\frac{1}{2})\Gamma(\frac{3}{2}-H)}\int_{0}^{\infty}\sqrt{Y_{\beta}}Q_{x}(t)\frac{dx}{x^{\frac{\alpha-1}{2}}}.$$
\qedhere	
\end{proof}

It is surely worthwhile to take a closer look at the process $Q_{x}$ in order to obtain a similar result as in Prop.~\ref{OU:rep1}. For this purpose we have to rewrite the process $Q_x$ at first in terms of a Wiener integral. 
It follows from \eqref{eq:Qx} that $Q_x(t)$ can be written as:
$$Q_x(t)=\int_{\mathbb{R}} 1\!\!1_{[0,t)}(s) (t-s)e^{-x(t-s)}\,dW(s).$$
This consideration enables us to work out the density function in the case $1 < \alpha < 2$ as in Prop.~\ref{OU:rep1}. In particular we obtain:

\begin{prop}\label{OU:rep2}
The process
\[
W^{\beta}_x(t)=\sqrt{Y_{\beta}}Q_{x}(t),
\]
where $$dQ_x(t) = (-xQ_x(t) + X_{x}(t)) dt, $$
w.r.t.~ $X_{x}(t)$, i.e.~ the Ornstein-Uhlenbeck process defined in~\eqref{eq:OU_rough}
	has the characteristic function 
	$$\mathbb{E}\left(e^{i(k,W_{x}^{\beta}(t))}\right)= E_{\beta}\left(-\frac{|k|^{2}}{2}\left|g_{x}(t,\cdot)\right|^{2}\right),$$
	where $g_{x}(t,s)=1\!\!1_{[0,t)}(s)(t-s)e^{-x(t-s)}$.\\
	Moreover its density function for $t\geq 0$ is given by
	$$\rho_{x}(y,t):=\rho_{W_{x}^{\beta}(t)}(y)=\frac{1}{2}(4\pi)^{\nicefrac{d}{2}}\mathbb{M}_{\nicefrac{\beta}{2}}^{d}\left(\sqrt{2}y,|g_x(t,\cdot)|^{\frac{2}{\beta}}\right), \quad y \in \mathbb{R}^d.$$
\end{prop}

\begin{proof}
The proof is completely analogue to the computations in Prop.~\ref{OU:rep1}, substituting the density function $f_x$ by $g_x$. \qedhere
\end{proof}

\begin{rem}[Simulation of ggBm]
We wish to point out that the representation from the Theorems \ref{repthm1} and \ref{repthm2} are ill-suited for a fast simulation of ggBm paths. The simulation of a path of fractional Brownian motion using the representation is strongly related to the simulation of the direct Mandelbrot-van Ness formula, which is known as method only to be used for academic purposes, see e.g.~the survey \cite{Co00}. 
The reason for that lies in the fat-tail behavior in the Laplace parameter. Hence a fast simulation can be performed using the well-known Wood and Chan method \cite{WoodChan} for the fBm part and directly simulate $\sqrt{Y_{\beta}}$ using the Taylor expansion of the pdf. 
In addition, an effective method to simulate ggBm has been recently published in \cite{PagniniMetzler2018}.
\end{rem}

\section{Conclusions}
The representation of ggBm by infinitely many generalized grey Ornstein Uhlenbeck processes holds in terms of finite dimensional distributions. This is due to the fact that the product of $\sqrt{Y_{\beta}}$ with an fBm yields a ggBm only in finite dimensional distributions. The charme of this representation lies hence in the representation of fBm by infinitely many Ornstein-Uhlenbeck processes w.r.t.~Brownian motion.

This paves the way to apply these analytic methods for the study of grey stochastic differential equations and the further development of a tractable stochastic analysis for ggBm.

\subsection*{Acknowledgements}

We wish to thank the referee and the associate editor for constructive comments that helped to improve our manuscript. Financial support from FCT -- Funda{\c c\~a}o para a Ci{\^e}ncia e a Tecnologia through the project UID/MAT/04674/201{9} (CIMA Universidade da Madeira) and the DFG-Research Training Group 1932 \textit{Stochastic Models for Innovations in the Engineering Sciences} is gratefully acknowledged. S. Desmettre is also supported by the Austrian Science Fund (FWF) project F5508-N26, which is part of the Special Research Program \textit{Quasi-Monte Carlo Methods: Theory and Applications}. Moreover, the authors gratefully acknowledge support from NAWI Graz. 

\appendix

{
\section{The Stochastic Fubini Theorem} For the interchanging of stochastic and Lebesgue integrals in the proofs of Theorem~\ref{repthm1} and Theorem~\ref{repthm2} we refer to a suitable version of the stochastic Fubini Theorem as it is given in \cite{Veraar12} and as it has also been used in \cite{HS2018}:
Let $\mu$ be a $\sigma$-finite measure on $(0,\infty)$. For fixed $T\geq 0$ denote by $\mathcal{A}$ the $\sigma$-algebra on $[0,T] \times \Omega$ generated by all progressively measurable processes.
\begin{thm}\label{thm:Fubini}
Let $G:(0,\infty) \times [0,T] \times \Omega \to \R$ be measurable with respect to the product $\sigma$-algebra $\mathcal{B}(0,\infty) \otimes \mathcal{A}$. Moreover, define processes $\xi_{1,2}:(0,\infty) \times [0,T] \times \Omega \to \R$ and $\eta:[0,T]\times\Omega \to \R$ by
\begin{align*}
\xi_1(x,t,\omega) = \int_0^t \hspace{-0.15cm} G(x,s,\omega)ds\,,\, \xi_2(x,t,\omega) = \Bigl( \int_0^t \hspace{-0.15cm} G(x,s,\cdot)dW(s) \Bigr) (\omega)\,,\, \eta(t,\omega) = \int_0^\infty \hspace{-0.15cm} G(x,t,\omega)  \mu(dx).
\end{align*}
\begin{enumerate}
\item Assume $G$ satisfies for almost all $\omega \in \Omega$
\begin{align}\label{eq:Fub1}
\int_0^\infty \left(\int_0^T |G(x,s,\omega)| ds \right) \mu(dx) < \infty\,.
\end{align}
Then, for almost all $\omega \in \Omega$ and for all $t\in [0,T]$ we have $\xi_1(\cdot,t,\omega) \in L^1(\mu)$ and 
\begin{align*}
\int_0^\infty \xi_1(x,t,\omega) \mu(dx) = \int_0^t \eta(s,\omega) ds\,.
\end{align*}
\item Assume $G$ satisfies for almost all $\omega \in \Omega$
\begin{align}\label{eq:Fub2}
\int_0^\infty \left( \sqrt{\int_0^T G(x,s,\omega)^2 ds} \right) \mu(dx) < \infty\,.
\end{align}
Then, for almost all $\omega \in \Omega$ and for all $t\in [0,T]$ we have $\xi_2(\cdot,t,\omega) \in L^1(\mu)$ and 
\begin{align*}
\int_0^\infty \xi_2(x,t,\omega) \mu(dx) = \left(\int_0^t \eta(s,\cdot) dW(s)\right) (\omega)\,.
\end{align*}
\end{enumerate}
\end{thm}
\section{Integrability of Basic Expressions}
We moreover provide the following auxiliary result for the integrability of elementary expressions which is closely related to and in the spirit of Lemma~6.7 in \cite{HS2018}.
\begin{lem}\label{lem:basic}
Suppose that $\mu$ and $\nu$ are sigma-finite measures on $(0,\infty)$ such that 
\begin{align}
&\int_0^\infty \left(1\land x^{-\frac{1}{2}} \right) \, \mu(dx) < \infty \,,\label{eq:finite1}\\
&\int_0^\infty \left(1\land x^{-\frac{3}{2}} \right) \, \nu(dx) < \infty\,,\label{eq:finite2}
\end{align}
and let $\tau,\alpha > 0$. Then we have that
\begin{align}
&\int_0^\infty \sqrt{\frac{1-e^{-2\tau x}}{x}} \, \mu(dx) < \infty\,,\label{eq:integrability1}\\
&\int_0^\infty \sqrt{\frac{1-e^{-2\tau x}(1+2\tau x + 2 \tau^2 x^2)}{x^3}} \, \nu(dx) < \infty\,. \label{eq:integrability2}
\end{align}
\end{lem}
\begin{proof}
Using the elementary inequality (for the proof compare Lemma 6.6 in \cite{HS2018}) 
\begin{align*}
\frac{1-e^{-\tau x}}{x} \leq (1 \vee \tau) (1 \land x^{-1})
\end{align*}
we obtain \eqref{eq:integrability1} with the help of \eqref{eq:finite1} as follows:
\begin{align*}
\int_0^\infty \sqrt{\frac{1-e^{-2\tau x}}{x}} \, \mu(dx) \leq \left( 1 \vee (2\tau)^{\tfrac{1}{2}} \right) \int_0^\infty \left(  1 \land x^{-\tfrac{1}{2}}  \right) \, \mu(dx) < \infty\,.
\end{align*}\
In the same spirit, using the elementary inequality (compare again Lemma 6.6 in \cite{HS2018}) 
\begin{align*}
\frac{1-e^{-\tau x}(1+\tau x + \tfrac{1}{2} \tau^2 x^2)}{x^3} \leq \left( 1 \vee \tau^3 \right) \left( 1 \land x^{-3} \right)
\end{align*}
we obtain \eqref{eq:integrability2} using \eqref{eq:finite2} as follows:
\begin{align*}
\int_0^\infty \sqrt{\frac{1-e^{-2\tau x}(1+2\tau x + 2 \tau^2 x^2)}{x^3}} \, \nu(dx)  \leq \left( 1 \vee (2\tau)^{\tfrac{3}{2}} \right) \int_0^\infty \left(  1 \land x^{-\tfrac{3}{2}}  \right) \, \nu(dx) < \infty\,.
\end{align*}
\end{proof}
}

\end{document}